\documentclass[12pt]{amsart}
\usepackage{amsmath, amsthm, amssymb}
\usepackage{graphicx}
\usepackage[margin=1in]{geometry}
\usepackage[all]{xy}
\usepackage[hidelinks]{hyperref}

\newtheorem{theorem}{Theorem}[section]
\newtheorem{lemma}[theorem]{Lemma}
\newtheorem{proposition}[theorem]{Proposition}
\newtheorem{conjecture}[theorem]{Conjecture}
\newtheorem{corollary}[theorem]{Corollary}

\theoremstyle{definition}
\newtheorem{definition}[theorem]{Definition}
\newtheorem{remark}[theorem]{Remark}
\newtheorem{example}[theorem]{Example}

\DeclareMathOperator{\gr}{gr}
\DeclareMathOperator{\multi}{multi}
\DeclareMathOperator{\unknot}{unknot}

\newcommand{\Z}{\mathbb{Z}}
\newcommand{\Q}{\mathbb{Q}}

\newcommand{\Cc}{\mathcal{C}}
\newcommand{\Hc}{\mathcal{H}}

\newcommand{\CFK}{\mathit{CFK}}
\newcommand{\HFK}{\mathit{HFK}}

\newcommand{\gl}{\mathfrak{gl}}

\title[Evaluations of link polynomials]{Evaluations of link polynomials and recent constructions in Heegaard Floer theory}
\author{Larry Gu}
\email{ljgu@usc.edu}
\address{Department of Mathematics\\ University of Southern California \\ Los Angeles, CA}

\author{Andrew Manion}
\email{amanion@usc.edu}
\address{Department of Mathematics\\ University of Southern California \\ Los Angeles, CA}

\begin{document}

\begin{abstract}
Using a definition of Euler characteristic for fractionally-graded complexes based on roots of unity, we show that the Euler characteristics of Dowlin's ``$\mathfrak{sl}(n)$-like'' Heegaard Floer knot invariants $\HFK_n$ recover both Alexander polynomial evaluations and $\mathfrak{sl}(n)$ polynomial evaluations at certain roots of unity for links in $S^3$. We show that the equality of these evaluations can be viewed as the decategorified content of the conjectured spectral sequences relating $\mathfrak{sl}(n)$ homology and $\HFK_n$.
\end{abstract}

\maketitle

\section{Introduction}

Ozsv{\'a}th--Szab{\'o}'s theory of Heegaard Floer homology \cite{HFOrig} is a flexible set of constructions yielding many types of invariants for low-dimensional manifolds. Even for knots and links in $S^3$, the ideas of Heegaard Floer homology can be applied in several ways to produce a family of related invariants known collectively as knot Floer homology or $\HFK$ \cite{HFKOrig,RasmussenThesis}. The simplest variant, $\widehat{\HFK}$ as applied to (single-component) knots in $S^3$, assigns to a knot $K$ a bigraded vector space $\widehat{\HFK}(K)$ (say over $\Q$) whose graded Euler characteristic is the Alexander polynomial $\Delta_K(t)$. Here we focus on more recent variants $\widehat{\HFK}_n$ (which we will call $\overline{\HFK}_n$) and $\HFK_n$, due to Dowlin \cite{dowlin2018family}, with relationships to Khovanov homology and $\mathfrak{sl}(n)$ homology more generally.

Reduced and unreduced $\mathfrak{sl}(n)$ homology \cite{KR1}, like $\widehat{\HFK}$, also assign bigraded vector spaces to knots $K$. Their graded Euler characteristics are the reduced and unreduced $\mathfrak{sl}(n)$ polynomials of $K$. The $\mathfrak{sl}(n)$ polynomials and the Alexander polynomial are all specializations of the two-variable HOMFLY-PT polynomial of $K$, leading to various relationships between the $\mathfrak{sl}(n)$ and Alexander polynomials at special values.

For the homology theories at the categorified level, one can often think of these relationships between knot polynomial evaluations as being categorified by certain spectral sequences that are known or conjectured to exist. For instance, the appearance of the $\mathfrak{sl}(n)$ polynomial as an evaluation of the HOMFLY-PT polynomial is categorified by Rasmussen's spectral sequences \cite{rasmussen2006some} from triply graded HOMFLY-PT homology \cite{KR2} to $\mathfrak{sl}(n)$ homology; the appearance of the Alexander polynomial as a HOMFLY-PT evaluation should be categorified by the conjectured spectral sequence from HOMFLY-PT homology to $\HFK$ \cite{dunfield2006superpolynomial}. 

In general, given some construction or conjecture in the realm of $\mathfrak{sl}(n)$ homology or $\HFK$, it is natural to ask ``what does it categorify, if anything?''; in other words, ``what is its decategorified content?''. Often this is something simpler than what one started with; for example, the identities relating $\mathfrak{sl}(n)$ polynomials and Alexander polynomials with the HOMFLY-PT polynomial are simpler than the known and conjectured spectral sequences from HOMFLY-PT homology to $\mathfrak{sl}(n)$ homology and $\HFK$. Investigating the decategorified level can be an easy way to gain valuable information about the structure one expects at the categorified level.

For this reason, it is natural to ask about the decategorified content of Dowlin's conjectured spectral sequences \cite[Conjecture 1.6]{dowlin2018family} from reduced and unreduced $\mathfrak{sl}(n)$ homology to $\HFK$ (generalizing Rasmussen's conjecture for $\mathfrak{sl}(2)$, proved by Dowlin in \cite{dowlin2018spectral}). The specific variants of $\HFK$ appearing in the conjectured spectral sequences are Dowlin's singly-graded variants $\overline{\HFK}_n$ and $\HFK_n$ for links (the first of these agrees with a grading collapse of $\widehat{\HFK}$ when applied to knots).

One complication is that as defined, $\overline{\HFK}_n$ and $\HFK_n$ are the homology of complexes whose differentials increase the single grading $\gr_n$ by $n$. The usual Euler characteristic formula, applied to such a complex, will not always be homotopy invariant. Instead, we divide the grading on $\overline{\HFK}_n$ and $\HFK_n$ by $n$, producing $\frac{1}{n}\Z$-graded complexes whose differentials increase the grading by one, and work with a natural generalization of the Euler characteristic to this setting (based on roots of unity and admitting an interpretation in terms of Grothendieck groups of triangulated categories). 

\begin{theorem}\label{thm:IntroHFKn}
Let $L$ be an $\ell$-component link in $S^3$ and let $n \geq 2$. The Euler characteristic of $\frac{\gr_n}{n}$-graded $\overline{\HFK}_n(L)$, in our sense, equals $e^{\pi i (1 - \ell) / n} \Delta_L(t)|_{t^{1/2} = -e^{- \pi i / n}}$ where $\Delta_L(t)$ is the symmetric single-variable Alexander polynomial of $L$ (a Laurent polynomial in $t^{1/2}$), and the Euler characteristic of $\HFK_n(L)$ is zero. For $n = 1$, the Euler characteristics of both $\overline{\HFK}_1(L)$ and $\HFK_1(L)$ are $1$ for all links $L$.
\end{theorem}

We introduce grading-modified versions $\overline{\HFK}'_n(L)$ and $\HFK'_n(L)$ of Dowlin's invariants such that the Euler characteristic of $\overline{\HFK}'_n(L)$ equals $\Delta_L(t)|_{t^{1/2} = -e^{\pi i / n}}$. These $n$-dependent invariants are related to bigraded versions $\overline{\HFK}'(L)$ and $\HFK'(L)$ categorifying $\Delta_L(t)$ and zero respectively.\footnote{From the representation-theoretic perspective, $\overline{\HFK}'(L)$ categorifies the $U_q(\gl(1|1))$ invariant of a link with one component cut open to form a $(1,1)$-tangle, while $\HFK'(L)$ categorifies the $U_q(\gl(1|1))$ invariant of a closed link which is zero.}

We propose $\overline{\HFK}'_n(L)$ and $\HFK'_n(L)$ as the $E_{\infty}$ pages of Dowlin's conjectured spectral sequences, and support our proposal with Euler characteristic evidence. The $E_2$ pages of these conjectured spectral sequences should be reduced and unreduced $\mathfrak{sl}(n)$ homology with the bigrading collapsed to a single grading and divided by $n$ as above. By analogy with $\HFK_n$, we will refer\footnote{This usage of $\gr_n$ conflicts with the notation in \cite{rasmussen2006some}, where $\gr_n$ is used for what we call the quantum grading on $\mathfrak{sl}(n)$ homology (itself related in an $n$-dependent way to the quantum and horizontal gradings on HOMFLY-PT homology; see \cite{rasmussen2006some}).} to the single collapsed grading as $\gr_n$ and its quotient by $n$ as $\frac{\gr_n}{n}$.

\begin{theorem}\label{thm:IntroSLn}
Let $L$ be a link in $S^3$. The Euler characteristic of the $\frac{\gr_n}{n}$-graded reduced $\mathfrak{sl}(n)$ homology of $L$, in our sense, equals the reduced $\mathfrak{sl}(n)$ polynomial of $L$ evaluated at $q = e^{\pi i / n}$. The Euler characteristic of the $\frac{\gr_n}{n}$-graded unreduced $\mathfrak{sl}(n)$ homology of $L$ equals the unreduced $\mathfrak{sl}(n)$ polynomial of $L$ evaluated at $q = e^{\pi i / n}$.
\end{theorem}

For $n \geq 2$ the polynomial evaluations appearing in Theorem~\ref{thm:IntroSLn} equal $\Delta_L(t)|_{t^{1/2} = - e^{\pi i / n}}$ and zero respectively; indeed, in the reduced case both evaluations are equal to $\overline{P}_L(-1,e^{\pi i / n})$ where $\overline{P}_L(a,q)$ is the reduced HOMFLY-PT polynomial of $L$, and similarly in the unreduced case. When $n = 1$ both evaluations are equal to $1$. As we discuss below, we can view Dowlin's conjectured spectral sequences as categorifications of these equalities.

We situate these results in the context of Rasmussen's spectral sequences from HOMFLY-PT homology to $\mathfrak{sl}(n)$ homology and the conjectured spectral sequences from HOMFLY-PT homology to $\HFK$, which fit with Dowlin's conjectured spectral sequences into a square as shown in \cite[Figure 1]{dowlin2018family}. We review the decategorified content of the known and conjectured spectral sequences starting at HOMFLY-PT homology, which we generalize to links in terms of our shifted gradings, and we add to Dowlin's square by labeling the edges with their decategorified content (see Figure~\ref{fig:SpectralSeqs}). Examining the decategorified content along the possible paths in the square, in terms of link polynomial evaluations, reveals a compatibility that could be a sign of a more structured relationship between these spectral sequences at the categorified level.

\begin{example}
Dowlin computes $\HFK_n$ of the unknot in \cite[Example 2.10]{dowlin2018family}; the result is $\Q[U]/(U^n)$, and the generator $1 \in \Q[U]$ has $\gr_n$ equal to $1 - n$ (so it has $\frac{\gr_n}{n}$ equal to $\frac{1}{n} - 1$). Dowlin writes that the graded Euler characteristic of this homology is $\frac{q^n - q^{-n}}{q - q^{-1}}$, agreeing with the $\mathfrak{sl}(n)$ polynomial of the unknot.

We propose that $\frac{q^n - q^{-n}}{q - q^{-1}}$ is the $\gr_n$-graded Poincar{\'e} polynomial of this homology group, where the coefficient of $q^i$ in the polynomial is the dimension of the homology in $\gr_n = i$. In fact, $\HFK_n$ of the unknot is isomorphic to unreduced $\mathfrak{sl}(n)$ homology of the unknot, which is naturally bigraded; in this example, the homological component of the bigrading is zero and the intrinsic component agrees with $\gr_n$. With respect to this bigrading, which makes sense on $\HFK_n$ of the unknot but not on $\HFK_n$ in general, it is indeed true that the graded Euler characteristic of $\HFK_n$ of the unknot is $\frac{q^n - q^{-n}}{q - q^{-1}}$.

However, here we are considering $\gr_n$ (divided by $n$), which makes sense on $\HFK_n$ of arbitrary links, as a homological grading. Since it is only a single grading, its Euler characteristic in our sense will be a single complex number (not necessarily an integer because the grading is fractional). A generator of the homology in $\gr_n = k$ will contribute a term $e^{k \pi i / n}$ to the Euler characteristic by our definitions; the Euler characteristic for the unknot homology is thus
\[
e^{\pi i (1/n - 1)} + e^{\pi i (3/n - 1)} + \cdots + e^{\pi i (1 - 3/n)} + e^{\pi i (1 - 1/n)}.
\]
For $n = 1$ we get $e^{\pi i (0)} = 1$; for $n \geq 2$ the sum is zero, since the roots of unity are distributed symmetrically around the origin.
\end{example}

\begin{remark}
Here we see chain complexes with gradings by $\frac{1}{n}\Z$ (and $d^2 = 0$) categorifying evaluations of $\mathfrak{sl}(n)$ polynomials at $2n^{th}$ roots of unity. For categorification of these polynomials at roots of unity $e^{2 \pi i / p}$ for $p$ prime, complexes with $d^2 = 0$ are no longer suitable, and one often works with $p$-complexes satisfying $d^p = 0$ (see e.g. \cite{KhHopfological,QiHopfological}). Combining these ideas, one could look for $p$-complexes with gradings by $\frac{1}{n}\Z$ categorifying evaluations of $\mathfrak{sl}(n)$ polynomials at $pn^{th}$ roots of unity, although we are not aware of such complexes in the literature.
\end{remark}

\begin{remark}
If one is only interested in categorifying e.g. $\overline{P}_{n,L}(e^{\pi i / n})$ using the ideas of this paper where $\overline{P}_{n,L}(q)$ is the $\mathfrak{sl}(n)$ polynomial, one does not need to use $\overline{\HFK}_n$; it suffices to take a grading-collapse of reduced $\mathfrak{sl}(n)$ homology. However, $\overline{\HFK}_n$ is a more natural or minimal categorification of this evaluation; analogously, to categorify the Alexander polynomial one can take a grading collapse of HOMFLY-PT homology, but $\overline{\HFK}$ is a more minimal way to do it.
\end{remark}

\begin{remark}
Let $K$ be a knot. For $n = 2$ where a spectral sequence from Khovanov homology to $\HFK$ has been constructed by Dowlin \cite{dowlin2018spectral}, the equality $\overline{P}_{2,K}(i) = \Delta_K(-1)$ is familiar (both evaluations give the knot determinant) and is a sign of a deeper relationship between the representation theory of $U_q(\gl(1|1))$ and $U_q(\gl(2))$ at $q = i$; see \cite[Section 1]{KauffmanSaleur}. We do not know whether there is any similar story one can tell about the analogous equalities for $n > 2$, although the $n = 2$ case is special at least in that both $\gl(1|1)$ and $\gl(2)$ are defined using $2 \times 2$ matrices.
\end{remark}

\subsection*{Organization}

In Section~\ref{sec:AlgPrelims} we define Euler characteristics for fractionally graded complexes and discuss spectral sequences. In Section~\ref{sec:LinkPolysKR} we review what we need about HOMFLY-PT polynomials, $\mathfrak{sl}(n)$ polynomials, and Alexander polynomials as well as HOMFLY-PT homology and $\mathfrak{sl}(n)$ homology (focusing on the gradings). In Section~\ref{sec:HFK} we do the same for $\HFK$ while introducing bigrading-shifted versions of $\HFK$ theories adapted to the three variants of HOMFLY-PT homology. In Section~\ref{sec:HFKnDefs} we recall the definitions of Dowlin's $\HFK_n$ invariants; in Section~\ref{sec:EulerChar} we compute their fractionally-graded Euler characteristics and introduce grading-shifted variants of $\HFK_n$. In Section~\ref{sec:SpectralSequences} we state a version of Dowlin's spectral sequence conjectures involving grading-shifted $\HFK_n$, compute its decategorified content, and place it in the context of spectral sequences from HOMFLY-PT homology to $\mathfrak{sl}(n)$ homology and $\HFK$.

\subsection*{Acknowledgments} 

We would like to thank Aaron Lauda for useful conversations and suggestions. A.M. was partially supported by NSF grant DMS-1902092 and Army Research Office W911NF2010075.

\section{Algebraic preliminaries}\label{sec:AlgPrelims}

\subsection{Euler characteristics of fractionally-graded complexes}

Following \cite{dowlin2018family}, we will work over $\Q$.

\begin{definition}
For $n \geq 1$, a $\frac{1}{n}\Z$-graded complex $C$ of $\Q$-vector spaces (or just a $\frac{1}{n}\Z$-graded complex for short) is a $\frac{1}{n}\Z$-graded $\Q$-vector space
\[
C = \bigoplus_{\alpha \in \frac{1}{n}\Z} C_{\alpha}
\]
equipped with a $\Q$-linear endomorphism $d$ of degree $+1$ satisfying $d^2 = 0$.
\end{definition}

\begin{remark}
A $\frac{1}{n}\Z$-graded complex is the same data as $n$ ordinary complexes, one for each element of $(\frac{1}{n}\Z)/\Z$. However, the examples of interest here more naturally give a $\frac{1}{n}\Z$-graded complex than $n$ ordinary complexes.
\end{remark}

The category of $\frac{1}{n}\Z$-graded complexes and homotopy classes of degree-zero chain maps is triangulated; the translation functor is degree shift downward by one. Furthermore, degree shift downward by $\frac{1}{n}$ equips this triangulated category with an $n^{th}$ root of its translation functor. We let $C[\alpha]$ denote $C$ with its degrees shifted downward by $\alpha$, so that $(C[\alpha])_{\alpha'} = C_{\alpha' + \alpha}$; the usual notation $[1]$ for the translation functor of a triangulated category agrees with our notation.

\begin{definition}\label{def:K0Fractional}
Let $\Cc$ be an (essentially small) triangulated category equipped with an $n^{th}$ root $\left[\frac{1}{n}\right]$ of its translation functor $[1]$. Let $\zeta_n = e^{\pi i / n}$. We define the Grothendieck group $K_0(\Cc)$ to be the quotient of the free $\Z[\zeta_n]$-module spanned by isomorphism classes of objects of $\Cc$ by the relations $X - Y + Z = 0$ for every distinguished triangle $X \to Y \to Z \to X[1]$ in $\Cc$, as well as 
\[
\left[X\left[\frac{1}{n}\right]\right] = \zeta_n^{-1}  \left[X\right]
\]
for all objects $X$ of $\Cc$.
\end{definition}

We can apply Definition~\ref{def:K0Fractional} to the homotopy category $H$ of finite-dimensional $\frac{1}{n}\Z$-graded complexes (a full triangulated subcategory of the homotopy category of all $\frac{1}{n}\Z$-graded complexes, preserved by the $n^{th}$ root of the translation functor). The result is a free $\Z[\zeta_n]$-module $K_0(H)$ of rank $1$ spanned by $[\Q]$, where $[\Q]$ denotes the class of the complex that has $\Q$ in degree zero and zero in all other degrees.

\begin{definition}
Let $C$ be a finite-dimensional $\frac{1}{n}\Z$-graded complex. The Euler characteristic $\chi(C)$ of $C$ is the unique element of $\Z[\zeta_n]$ such that $[C] = \chi(C) [\Q]$ in $K_0(H)$. Explicitly,
\[
\chi(C) = \sum_{\alpha \in \frac{1}{n}\Z} e^{\pi i \alpha} \dim_{\Q} C_{\alpha}.
\]
\end{definition}

Just as for ordinary Euler characteristics, we have $\chi(C) = \chi(H_*(C))$. We also recall the usual graded Euler characteristics for bigraded and triply-graded complexes.

\begin{definition}\label{def:BigradedEulerChar}
Let $C = \left( \{C_{I,J} : I,J \in \Z\}, d \right)$ be a bigraded chain complex which is finite-dimensional in each $I$-degree, such that $d$ has degree $(0,-1)$ or $(0,1)$. The graded Euler characteristic of $C$ is defined to be
\[
\chi_u(C) = \sum_{I,J \in \Z} (-1)^J u^I \dim_{\Q} (C_{I,J}),
\]
a formal Laurent series in a variable $u$. If the $I$-grading is valued in $\frac{1}{2}\Z$ rather than $\Z$, the same definition gives a formal Laurent series in $u^{1/2}$.

Similarly, let $C = \left( \{ C_{I,J,K} : I,J,K \in \Z \}, d \right)$ be a triply graded chain complex which is finite-dimensional in each $(I,J)$-bidegree, such that $d$ has degree $(0,0,-1)$ or $(0,0,1)$. The graded Euler characteristic of $C$ is defined to be
\[
\chi_{u,v}(C) = \sum_{I,J,K \in \Z} (-1)^K u^I v^J \dim_{\Q} (C_{I,J,K}),
\]
a formal Laurent series in variables $u$ and $v$. 
\end{definition}

\begin{remark}
Rather than $u$ and $v$, we will often use variable names corresponding to the gradings in question.
\end{remark}

\subsection{Spectral sequences}

\begin{definition}\label{def:SpectralSeq}
A spectral sequence of $\frac{1}{n}\Z$-graded complexes is a sequence of $\frac{1}{n}\Z$-graded complexes $(E_r,d_r)_{r \geq 0}$ together with isomorphisms $H_*(E_r,d_r) \cong E_{r+1}$ for $r \geq 0$.
\end{definition}

All spectral sequences in this paper have $d_r = 0$ for large enough $r$, so that for some $\frac{1}{n}\Z$-graded vector space $E_{\infty}$ we have $(E_r,d_r) = (E_{\infty},0)$ for large enough $r$.

\begin{remark}
It is built into the above definition that the differential $d_r$ on each page of the spectral sequence has degree $+1$ with respect to the $\frac{1}{n}\Z$ grading.
\end{remark}

Since $\chi(C) = \chi(H_*(C))$ for finite-dimensional $\frac{1}{n}\Z$-graded chain complexes, if some page $E_r$ of a spectral sequence as in Definition~\ref{def:SpectralSeq} is finite-dimensional then for any $r' \geq r$ we have $\chi(E_r) = \chi(E_{r'})$; in particular, $\chi(E_r) = \chi(E_{\infty})$. 

\begin{remark}\label{rem:BigradedSS}
Suppose one has a spectral sequence of bigraded complexes as in Definition~\ref{def:BigradedEulerChar}, such that each $d_r$ has bidegree $(0,-1)$ or $(0,1)$. Suppose that some page $E_r$ is finite-dimensional in each $j$-degree; the same is then true for each page $E_{r'}$ for $r \geq r$, and it follows for the same reason as above that $\chi_u(E_r) = \chi_u(E_{r'})$. In particular, $\chi_u(E_r) = \chi_u(E_{\infty})$. A similar equality for $\chi_{u,v}$ holds in the triply graded case, assuming each $d_r$ has bidegree $(0,0,-1)$ or $(0,0,1)$.
\end{remark}

\section{Link polynomials and Khovanov--Rozansky homology}\label{sec:LinkPolysKR}

All links below are assumed to be oriented.

\subsection{Link polynomials}\label{sec:LinkPolys}

The HOMFLY-PT polynomial $P_L(a,q)$ of a link $L$ in $S^3$ \cite{HOMFLY,PT} is defined by the skein relation
\[
a P_{L_+}(a,q) - a^{-1} P_{L_-}(a,q) = (q - q^{-1}) P_{L_0}(a,q)
\]
(where $L_+$, $L_-$, and $L_0$ are related near a crossing as in Figure~\ref{fig:SkeinRel}) together with the HOMFLY-PT polynomial of the unknot as a normalization. We consider three variants:
\begin{itemize}
\item The reduced HOMFLY-PT polynomial $\overline{P}_L(a,q)$ has $\overline{P}_{\unknot}(a,q) = 1$.
\item The middle HOMFLY-PT polynomial $P^-_L(a,q)$ has $P^-_{\unknot}(a,q) = \frac{-1}{q-q^{-1}}$.
\item The unreduced HOMFLY-PT polynomial $P_L(a,q)$ has $P_{\unknot}(a,q) = \frac{a - a^{-1}}{q - q^{-1}}$.
\end{itemize}

\begin{remark}
In these variables the middle and unreduced HOMFLY-PT ``polynomials'' are rational functions in general, although they are Laurent polynomials in $a$ and $z = q - q^{-1}$.
\end{remark}

\begin{figure}
\includegraphics[scale=0.6]{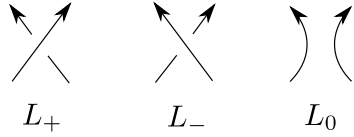}
\caption{Links appearing in the HOMFLY-PT skein relation.}
\label{fig:SkeinRel}
\end{figure}

We also consider three variants of the $\mathfrak{sl}(n)$ polynomial:
\begin{itemize}
\item The reduced $\mathfrak{sl}(n)$ polynomial is $\overline{P}_{n,L}(q) := \overline{P}_L(q^n,q)$.
\item The unreduced $\mathfrak{sl}(n)$ polynomial is $P_{n,L}(q) := P_L(q^n,q)$.
\end{itemize}

Finally, if we let $\Delta_L(t)$ denote the (symmetric single-variable) Alexander polynomial of $L$, a Laurent polynomial in $t^{1/2}$, then we have
\begin{itemize}
\item $\overline{P}_L(1,t^{1/2}) = \Delta_L(t)$,
\item $P^-_L(1,t^{1/2}) = \frac{\Delta_L(t)}{t^{-1/2} - t^{1/2}}$,
\item $P_L(1,t^{1/2}) = 0$.
\end{itemize}

More relevant for us will be the following identities, which are consequences of the symmetries $\overline{P}_L(a,q) = \overline{P}(-a,-q)$, $P^-_L(a,q) = -P^-_L(-a,-q)$, and $P_L(-a,-q) = P_L(a,q)$ of the HOMFLY-PT polynomials:
\begin{itemize}
\item $\overline{P}_L(-1,-t^{1/2}) = \Delta_L(t)$,
\item $P^-_L(-1,-t^{1/2}) = \frac{\Delta_L(t)}{t^{1/2} - t^{-1/2}}$,
\item $P_L(-1,-t^{-1/2}) = 0$.
\end{itemize} 

\subsection{Khovanov--Rozansky homology}

\subsubsection{Gradings and Euler characteristics}

We briefly establish notation for the $\mathfrak{sl}(n)$ homology and HOMFLY-PT homology of Khovanov--Rozansky \cite{KR1, KR2}; see also \cite{rasmussen2006some}. Let $\overline{H}(L)$, $H^-(L)$, and $H(L)$ be the reduced, middle, and unreduced HOMFLY-PT homology of a link $L$ in $S^3$. In the notation of \cite{rasmussen2006some}, these variants of HOMFLY-PT homology (denoted there by $\overline{H}(L)$, $H(L)$, and $\widetilde{H}(L)$ respectively) have a $\Z$-grading $\gr_q$ (or just $q$), a $\frac{1}{2}\Z$-grading $\gr_h$, and a $\frac{1}{2}\Z$-grading $\gr_v$. Rasmussen also writes $i = \gr_q$, $j = 2\gr_h$, and $k = 2\gr_v$; the value of $j-k$ is always even. We let
\begin{itemize}
\item $\gr_A = 2 \gr_h = j$,
\item $\gr_Q = \gr_q = i$,
\item $\gr_H = \gr_v - \gr_h = \frac{k-j}{2}$,
\end{itemize}
each of which is a grading by $\Z$ on the above three variants of HOMFLY-PT homology.

\begin{remark}
While $h$ in $\gr_h$ stands for horizontal, $H$ in $\gr_H$ stands for homological.
\end{remark}

Each variant of HOMFLY-PT homology is finite-dimensional in each $(\gr_A, \gr_Q)$-bidegree, so the following proposition makes sense.

\begin{proposition}[cf. Theorem 2.11, Section 2.8 of \cite{rasmussen2006some}]
For a link $L$ in $S^3$, we have:
\begin{itemize}
\item $\chi_{a,q}(\overline{H}(L)) = \overline{P}_L(a,q)$,
\item $\chi_{a,q}(H^-(L)) = P^-_L(a,q)$,
\item $\chi_{a,q}(H(L)) = P_L(a,q)$.
\end{itemize}
\end{proposition}

Now let $\overline{H}_n(L)$ and $H_n(L)$ be the reduced and unreduced $\mathfrak{sl}(n)$ homology of a link $L$ in $S^3$; the reduced homology $H_n(L)$ also depends on a choice of component of $L$. Both variants of $\mathfrak{sl}(n)$ homology have $\Z$-gradings $\gr_{Q, n}$ and $\gr_{H}$; in \cite[Section 2.9]{rasmussen2006some} these gradings are called $\gr_n$ and $\gr_v$ respectively, while in \cite[Section 5]{rasmussen2006some} they are called $\gr'_n$ and $\gr_-$ (see \cite[Proposition 5.14]{rasmussen2006some}).

\begin{proposition}[cf. Theorem 2.16 of \cite{rasmussen2006some}]
Write $\chi_q$ for the $\gr_{Q, n}$-graded Euler characteristic, with $\gr_H$ treated as the homological grading. For a link $L$ in $S^3$, we have:
\begin{itemize}
\item $\chi_q(\overline{H}_n(L)) = \overline{P}_{n,L}(q)$,
\item $\chi_q(H_n(L)) = P_{n,L}(q)$.
\end{itemize}
\end{proposition}

\subsubsection{Rasmussen's spectral sequences}\label{sec:RasmussenSeq}

In \cite{rasmussen2006some}, Rasmussen constructs spectral sequences with $E_2$ page $\overline{H}(L)$ (respectively $H(L)$) and $E_{\infty}$ page $\overline{H}_n(L)$ (respectively $H_n(L)$) for $n \geq 1$. As discussed in \cite[beginning of Section 5]{rasmussen2006some}, the differentials on each page have $\gr_{Q,n} = 0$ and $\gr_H = 1$ where $\gr_{Q,n}$ on HOMFLY-PT homology is defined by $\gr_{Q,n} = \gr_Q + n \gr_A$. Thus, each page gets a bigrading as the homology of the previous page, and the induced bigrading on the $E_{\infty}$ page agrees with $(\gr_{Q,n},\gr_H)$ on $\mathfrak{sl}(n)$ homology.

These spectral sequences give equalities of Euler characteristics $\chi_q(\overline{H}(L)) = \chi_q(\overline{H}_n(L))$ and $\chi_q(H(L)) = \chi_q(H_n(L))$, where we are viewing $\overline{H}(L)$ and $H(L)$ as bigraded by $(\gr_{Q,n}, \gr_H)$ and $\chi_q$ denotes the $\gr_{Q,n}$-graded Euler characteristic. As in the proof of \cite[Lemma 5.4]{rasmussen2006some}, we have $\chi_q(\overline{H}(L)) = \overline{P}_L(q^n, q)$; similarly, we have $\chi_q(H(L)) = P_L(q^n, q)$. Thus, applying Euler characteristics to these spectral sequences recovers the usual identity of the $\mathfrak{sl}(n)$ polynomial with an evaluation of the HOMFLY-PT polynomial. We view these identities as the ``decategorified content'' of Rasmussen's spectral sequences; in other words, we view the spectral sequences as categorifications of these identities.

\section{Knot Floer homology}\label{sec:HFK}

\subsection{The master complex}\label{sec:MasterComplex}

Let $L$ be a link in $S^3$; let $\Hc$ be a multi-pointed Heegaard diagram for $L$ (versions of the below theories can be defined for links in more general 3-manifolds but we restrict attention to links in $S^3$ here). Write 
\[
\{z_1,w_1,\ldots,z_m,w_m\}
\]
for the set of basepoints in $\Hc$. We assume that $\Hc$ is equipped with the appropriate analytic data such that the knot Floer homology ``master complex\footnote{While it would be more accurate to write $\CFK_{U,V}(\Hc)$, it will help avoid confusion with reduced and unreduced versions of knot Floer homology below to write $\CFK_{U,V}(L)$, with the $\Hc$ dependence left implicit.}'' $\CFK_{U,V}(L)$, a finitely generated bigraded free module over $\Q[U_1,V_1\ldots,U_m,V_m]$ with an endomorphism $\partial_{U,V}$ satisfying
\[
\partial_{U,V}^2 = \sum_{i=1}^m (U_{a(i)} - U_{b(i)}) V_i,
\]
is defined (here $a(i)$ denotes the index of the unique $w$ basepoint in the same component as $z_i$ of the Heegaard surface with alpha curves removed, and similarly for $b(i)$ and beta curves). See \cite{ZemkeFunctoriality, dowlin2018family} for more details on the master complex.

The two gradings on $\CFK_{U,V}(L)$ are called the Alexander grading $\gr_T$ (a grading by $\frac{1}{2}\Z$ in general) and the Maslov grading $\gr_M$ (a grading by $\Z$); our conventions for these gradings follow \cite{HFKOrig,OSzLinkInvts}. The variables $U_i$ have $\gr_T = -1$ and $\gr_M = -2$, the variables $V_i$ have $\gr_T = 1$ and $\gr_M = 0$, and $\partial_{U,V}$ has $\gr_T = 0$ and $\gr_M = -1$. 

In particular, to fix the absolute Maslov grading, one can work with a Heegaard diagram $\Hc$ whose number $m$ of $z$ and $w$ basepoints is equal to the number $\ell$ of components of $L$. Then the homology of
\[
\CFK_{U,V}(L) \otimes_{\Q[U_1,V_1,\ldots,U_{\ell},V_{\ell}]} \frac{\Q[U_1,V_1,\ldots,U_{\ell},V_{\ell}]}{(U_1, V_1 - 1, \ldots, V_{\ell} - 1)},
\]
which is a complex which $\partial^2 = 0$ with a single grading by $\gr_M$, computes $\widehat{\mathit{HF}}(S^3) \cong \Q$ (see the discussion after Theorem 4.4 of \cite{OSzLinkInvts}), while the homology of 
\[
\CFK_{U,V}(L) \otimes_{\Q[U_1,V_1,\ldots,U_{\ell},V_{\ell}]} \frac{\Q[U_1,V_1,\ldots,U_{\ell},V_{\ell}]}{(U_1, \ldots, U_{\ell}, V_1 - 1, \ldots, V_{\ell} - 1)}
\]
computes $\widehat{\mathit{HF}}(\#^{\ell - 1} (S^2 \times S^1)) \cong \wedge^*V$ where $V$ is a vector space of dimension $\ell - 1$. As mentioned in \cite[proof of Theorem 1.1]{OSzLinkInvts}, in the grading conventions of that paper the top-dimensional generator of $\wedge^*V$ corresponds to the generator of $\widehat{\mathit{HF}}(S^3) \cong \Q$. By the discussion after \cite[Theorem 1.2]{OSzLinkInvts}, the absolute Maslov grading on $\CFK$ is fixed so that the top-dimensional generator of $\wedge^* V$ has Maslov degree zero, so we can equivalently say that the generator of $\widehat{\mathit{HF}}(S^3) \cong \Q$ has degree zero. The absolute Alexander grading is fixed by symmetry.

\subsection{Other bigraded variants of \texorpdfstring{$\HFK$}{HFK}}

The following complexes are all derived from the master complex and satisfy $\partial^2 = 0$. We focus on bigraded versions of $\HFK$; there are also multi-graded versions as in \cite{OSzLinkInvts}. Let $L$ be an $\ell$-component link in $S^3$; when we mention knot Floer complexes for links, the dependence on a choice of Heegaard diagram for $L$ (say with basepoints $\{z_1,w_1,\ldots,z_m,w_m\}$) is implicit.

\begin{definition}\label{def:HFKTilde}
The bigraded complex $\widetilde{\CFK}(L)$ is
\[
\CFK_{U,V}(L) \otimes_{\Q[U_1,V_1,\ldots,U_m,V_m]} \Q.
\]
\end{definition}

\begin{definition}\label{def:HFKHat}
Assume that, in our Heegaard diagram $\Hc$ representing $L$, we are given some choice of basepoints $(z_{i_j},w_{i_j})$ on each component $L_j$ of $L$. The bigraded complex $\widehat{\CFK}(L)$ is 
\[
\CFK_{U,V}(L) \otimes_{\Q[U_1,V_1,\ldots,U_m,V_m]} \frac{\Q[U_1,V_1,\ldots,U_m,V_m]}{(V_1,\ldots,V_m,U_{i_1},\ldots,U_{i_\ell})}.
\]
\end{definition}

\begin{definition}\label{def:ReducedHFK}
Assume that $L$ is equipped with a distinguished component and that the basepoints $z_m, w_m$ of the Heegaard diagram $\Hc$ representing $L$ lie on the distinguished component of $L$. The bigraded complex $\overline{\CFK}(L)$ is
\[
\CFK_{U,V}(L) \otimes_{\Q[U_1,V_1,\ldots,U_m,V_m]} \frac{\Q[U_1,V_1,\ldots,U_m,V_m]}{(V_1,\ldots,V_m,U_m)}.
\]
\end{definition}

\begin{definition}\label{def:HFKMinus}
The bigraded complex $\CFK^-(L)$ is 
\[
\CFK_{U,V}(L) \otimes_{\Q[U_1,V_1,\ldots,U_m,V_m]} \frac{\Q[U_1,V_1,\ldots,U_m,V_m]}{(V_1,\ldots,V_m)}.
\]
\end{definition}

\begin{definition}\label{def:UnreducedHFK}[cf. Section 2.6 of \cite{dowlin2018family}]
Let $L'$ be the disjoint union of $L$ with a split unknot; we choose the unknot component as a distinguished component for $L'$, and we assume that the only basepoints of the diagram $\Hc'$ we choose to represent $L'$ that lie on the distinguished component of $L'$ are the final pair $(z_{m'},w_{m'})$ of basepoints. We define
\[
\CFK(L) := t^{-1/2} \overline{\CFK}(L')
\]
where $t^{-1/2}$ denotes a downward shift by $\frac{1}{2}$ in the Alexander grading $\gr_T$.
\end{definition}

\begin{remark}
The use of a split unknot to define unreduced $\HFK$ appears in Baldwin--Levine--Sarkar \cite{BaldwinLevineSarkar}, although these authors use the term ``unreduced $\HFK$'' for a slightly different theory.
\end{remark}

The homology of each of these complexes is an invariant of $L$ (equipped with a distinguished component in Definition~\ref{def:ReducedHFK}) and will be denoted by $\widetilde{\HFK}(L)$, $\widehat{\HFK}(L)$, etc. Each of the above bigraded versions of $\HFK$ is finite-dimensional in each Alexander degree.

\begin{remark}
The complex $\overline{\CFK}(L)$ and its homology appear to be less common in the literature; we use $\overline{(\cdot)}$ to match the notation of HOMFLY-PT homology and $\mathfrak{sl}(n)$ homology, although it is possible that our use of the notation $\overline{\HFK}(L)$ conflicts with uses of this notation elsewhere.
\end{remark}

\subsection{Graded Euler characteristics}

\begin{proposition}\label{prop:BigradedHFKEuler}
For a link $L$ in $S^3$, we have:
\begin{itemize}
\item $\chi_t(\widehat{\HFK}(L)) = (t^{-1/2} - t^{1/2})^{\ell - 1} \Delta_L(t) = (-1)^{\ell - 1} t^{\frac{\ell - 1}{2}} (1-t^{-1})^{\ell - 1} \Delta_L(t)$,
\item $\chi_t(\overline{\HFK}(L)) = (-1)^{\ell - 1} t^{\frac{\ell - 1}{2}} \Delta_L(t)$,
\item $\chi_t(\HFK^-(L)) = (-1)^{\ell - 1} t^{\frac{\ell - 1}{2}} \frac{\Delta_L(t)}{1-t^{-1}}$,
\item $\chi_t(\HFK(L)) = 0$.
\end{itemize}
\end{proposition}

\begin{proof}
The first claim follows from \cite[equation (1)]{HFKOrig} and \cite[Theorem 1.1]{OSzLinkInvts}.\footnote{Note that \cite[equation (1)]{OSzLinkInvts}, when proved in \cite[Proposition 9.1]{OSzLinkInvts}, is stated with a $\pm$ sign.} The second claim follows from the first because the chain groups in $\overline{\CFK}(L)$ are free modules over polynomial rings in $\ell - 1$ more variables than the corresponding chain groups in $\widehat{\CFK}(L)$; the graded Euler characteristic of a polynomial ring in one of these variables (with degrees $A = -1$ and $M = -2$) is $\frac{1}{1-t^{-1}}$. The third claim follows similarly; the fourth claim follows from the second claim along with the fact that $\Delta_L(t)$ vanishes on split links.
\end{proof}

\begin{remark}
For an $\ell$-component link $L$ in $S^3$, the single-variable and multi-variable Alexander polynomials of $L$ satisfy the relation $\Delta_L(t) \doteq \begin{cases} \Delta_L^{\multi}(t) & \ell = 1 \\ \Delta_L^{\multi}(t,\ldots,t)(1-t) & \ell > 1 \end{cases}$, where $\doteq$ means equality up to multiplication by a unit in $\Z[t,t^{-1}]$ (see \cite[Proposition 7.3.10(1)]{Kawauchi}). This relation explains why, unlike in \cite[equations (1) and (2)]{OSzLinkInvts}, we do not need to treat $\ell = 1$ and $\ell > 1$ separately in Proposition~\ref{prop:BigradedHFKEuler}.
\end{remark}

\begin{definition}\label{def:GradingShiftedHFK}
To more closely match the three variants of HOMFLY-PT homology, we introduce grading-shifted variants $\overline{\HFK}'(L)$, $(\HFK^-)'(L)$, and $\HFK'(L)$ of knot Floer homology. We first replace the Alexander and Maslov degrees of $\overline{\HFK}(L)$, $\HFK^-(L)$, and $\HFK(L)$ by their negatives, so that the variable $U_i$ now has $\gr_A = 1$ and $\gr_M = 2$ and the differentials on $\CFK$ complexes now have Maslov degree $+1$. In the Euler characteristic computations of Proposition~\ref{prop:BigradedHFKEuler}, $t$ gets replaced by $t^{-1}$; note that $\Delta_L(t^{-1}) = (-1)^{\ell - 1}\Delta_L(t)$. We then make the following shifts:
\begin{itemize}
\item For $\overline{\HFK}'(L)$, we shift the Alexander grading on grading-reversed $\overline{\HFK}(L)$ upward by $\frac{\ell - 1}{2}$. We have 
\[
\chi_t(\overline{\HFK}'(t)) = \Delta_L(t).
\]
\item For $(\HFK^-)'(L)$, we shift the Alexander grading on grading-reversed $\HFK^-(L)$ upward by $\frac{\ell}{2}$; we also shift the Maslov grading upward by $1$. We have 
\[
\chi_t((\HFK^-)'(L)) = \frac{\Delta_L(t)}{t^{1/2} - t^{-1/2}}.
\]
\item For $\HFK'(L)$, we shift the Alexander grading on grading-reversed $\HFK(L)$ upward by $\frac{\ell - 1}{2}$ (recall that $\HFK(L)$ already had an Alexander grading shift in Definition~\ref{def:UnreducedHFK}). We have $\chi_t(\HFK'(L)) = 0$.
\end{itemize}
\end{definition}

\subsection{Conjectured spectral sequences from HOMFLY-PT homology to \texorpdfstring{$\HFK$}{HFK}}\label{sec:HOMFLYtoHFK}

In \cite{dunfield2006superpolynomial}, Dunfield--Gukov--Rasmussen conjectured the existence of spectral sequences from $\overline{H}(K)$ to $\widehat{\HFK}(K)$ for knots $K$ in $S^3$. Manolescu \cite{ManolescuCube} gives a similar conjecture for $\HFK^-$ and the middle HOMFLY-PT homology. Dowlin \cite{dowlin2018family} conjectures spectral sequences from $\overline{H}(L)$ to $\widehat{\HFK}(L)$ and from $H(L)$ to $\HFK(L)$ for all links in $S^3$. We believe a spectral sequence involving $\overline{\HFK}(L)$ is more plausible in the reduced case for links, so we will state the following version of these spectral sequence conjectures.

\begin{conjecture}
Let $L$ be a link in $S^3$. Ignoring gradings at first, there are spectral sequences with:
\begin{itemize}
\item $E_2$ page $\overline{H}(L)$ and $E_{\infty}$ page $\overline{\HFK}(L)$;
\item $E_2$ page $H^-(L)$ and $E_{\infty}$ page $\HFK^-(L)$;
\item $E_2$ page $H(L)$ and $E_{\infty}$ page $\HFK(L)$.
\end{itemize}
Moreover, such sequences are given by the construction of Manolescu \cite[Theorem 1.1]{ManolescuCube}, which is known to give $E_{\infty}$ pages recovering $\HFK$.
\end{conjecture}

\begin{remark}
In \cite{DowlinSingular}, Dowlin identifies the $E_1$ page of the spectral sequence from \cite[Theorem 1.1]{ManolescuCube} with the appropriate sum of HOMFLY-PT complexes for singular resolutions of $L$; it remains to identify the $E_2$ page with HOMFLY-PT homology for links with nonsingular crossings.
\end{remark}

Manolescu \cite[Section 4]{ManolescuCube} discusses the grading properties of his conjectured spectral sequences from HOMFLY-PT homology to $\HFK$ in detail; we will rephrase some of his discussion in terms of the grading-shifted variants of $\HFK$ from Definition~\ref{def:GradingShiftedHFK}. We define a $\frac{1}{2}\Z$-grading $\gr_T$ and a $\Z$-grading $\gr_M$ on each of the variants of HOMFLY-PT homology by 
\begin{itemize}
\item $\gr_T = \frac{\gr_Q}{2}$,
\item $\gr_M = \gr_A + \gr_Q + \gr_H$
\end{itemize}
(Manolescu includes constant grading-shift terms in the above formulas but here we incorporate the grading shifts into $\HFK$; he also has negative signs since, unlike us, he has not multiplied the Alexander and Maslov gradings on $\HFK$ by $-1$). The differential $d_r$ on the $E_r$ page of Manolescu's conjectured sequences has $\gr_A = 2-2r$, $\gr_Q = 0$, and $\gr_H = 2r-1$. Thus, $d_r$ has $\gr_T = 0$ and $\gr_M = 1$. Writing $\overline{H}(L) = \oplus_{i,j,k \in \Z} \overline{H}^{i,j,k}(L)$ as in \cite{rasmussen2006some} (and similarly for the other versions), we equivalently have $\gr_T = i/2$ and $\gr_M = i+j/2+k/2$. 

\begin{conjecture}
Let $L$ be a link in $S^3$. There are spectral sequences with each page bigraded by $(\gr_T, \gr_M)$, such that the differential on each page has $(\gr_T, \gr_M) = (0,1)$ and each page is the bigraded homology of the previous page, and with
\begin{itemize}
\item $E_2$ page $\overline{H}(L)$ and $E_{\infty}$ page $\overline{\HFK}'(L)$;
\item $E_2$ page $H^-(L)$ and $E_{\infty}$ page $(\HFK^-)'(L)$;
\item $E_2$ page $H(L)$ and $E_{\infty}$ page $\HFK'(L)$
\end{itemize}
as bigraded vector spaces. 
\end{conjecture}

These spectral sequences would give equalities of Euler characteristics 
\[
\chi_t(\overline{H}(L)) = \chi_t(\overline{\HFK}'(L)), \quad \chi_t(H^-(L)) = \chi_t((\HFK^-)'(L)), \quad \chi_t(H(L)) = \chi_t(\HFK'(L))
\]
where $\chi_T$ denotes the $\gr_T$-graded Euler characteristic. We have
\begin{align*}
\chi_t(\overline{H}(L)) &= \sum_{I \in \frac{1}{2}\Z, J \in \Z} (-1)^J t^I \dim_{\Q} \left( \overline{H}(L)_{\gr_T = I, \gr_M = J} \right) \\
&= \sum_{i,j,k \in \Z} (-1)^{i+j/2+k/2}t^{i/2} \dim_{\Q} \left( \overline{H}^{i,j,k}(L) \right) \\
&= \left( \sum_{i,j,k \in \Z} a^j q^i (-1)^{(k-j)/2} \dim_{\Q} \left( \overline{H}^{i,j,k}(L) \right)\right)\bigg|_{a=-1, \,\, q=-t^{1/2}} \\
&= \overline{P}_L(-1,-t^{1/2}).
\end{align*}
Similarly,
\[
\chi_t(H^-(L)) = P^-_L(-1,-t^{1/2}), \qquad \chi_t(H(L)) = P_L(-1,-t^{1/2}) = 0.
\]
Thus, these conjectured spectral sequences can be viewed as categorifications of the three equalities involving Alexander polynomials and HOMFLY-PT polynomial evaluations (with $a = -1$) at the end of Section~\ref{sec:LinkPolys}.

\section{Dowlin's \texorpdfstring{$\HFK_n$}{HFKn} invariants}\label{sec:HFKnDefs}

We now consider two versions of $\HFK$ defined by Dowlin \cite{dowlin2018family}, applied to links in $S^3$ rather than more general 3-manifolds. Rather than bigradings, these versions will have single gradings by $\frac{1}{n}\Z$ in our conventions. 

Let $L$ be a link in $S^3$, represented by a Heegaard diagram $\Hc$ as in Section~\ref{sec:MasterComplex} with basepoints $\{z_1, w_1, \ldots, z_m, w_m\}$. Following Dowlin, for $n \geq 1$ we consider a collapse $\gr_n$ of the bigrading on $\CFK_{U,V}(L)$ defined by
\[
\gr_n = -n \gr_M + 2(n-1) \gr_T.
\]
We divide $\gr_n$ by $n$ to get
\[
\frac{\gr_n}{n} = -\gr_M + 2\left(1-\frac{1}{n}\right)\gr_T
\]
which is valued in $\frac{1}{n}\Z$ even for half-integral values of $\gr_T$. The variables $U_i$ have $\frac{\gr_n}{n} = \frac{2}{n}$,  the variables $V_i$ have $\frac{\gr_n}{n} = 2 - \frac{2}{n}$, and $\partial_{U,V}$ has $\frac{\gr_n}{n} = 1$.

\begin{definition}[cf. Definition 2.19 of \cite{dowlin2018family}]
Assume that $L$ is equipped with a distinguished component and that the basepoints $z_m$, $w_m$ of the Heegaard diagram $\Hc$ representing $L$ lie on the distinguished component. The $\frac{1}{n}\Z$-graded complex $\overline{\CFK}_n(L)$ is 
\[
\CFK_{U,V}(L) \otimes_{\Q[U_1,\ldots,U_m,V_1,\ldots,V_m]} \frac{\Q[U_1,\ldots,U_m,V_1,\ldots,V_m]}{\left(V_i - \frac{U_{a(i)}^n - U_{b(i)}^n}{U_{a(i)} - U_{b(i)}} : 1 \leq i \leq m-1 \right) + (U_m,V_m)}
\]
where $a(i)$ and $b(i)$ are defined as in Section~\ref{sec:MasterComplex}. The grading is given by $\frac{\gr_n}{n}$; note that $\frac{U_{a(i)}^n - U_{b(i)}^n}{U_{a(i)} - U_{b(i)}}$ equals the telescoping sum $U_{a(i)}^{n-1} + U_{a(i)}^{n-2} U_{b(i)} + \cdots + U_{b(i)}^{n-1}$, which (like $V_i$) has $\frac{\gr_n}{n} = 2 - \frac{2}{n}$.
\end{definition}

\begin{definition}[cf. Definition 2.5 of \cite{dowlin2018family}]
Let $L'$ be the disjoint union of $L$ with a split unknot, and choose the unknot component to be distinguished. As in Definition~\ref{def:UnreducedHFK}, we assume that the only basepoints of the diagram $\Hc'$ we choose to represent $L'$ that lie on the distinguished component of $L'$ are the final pair $(z_{m'},w_{m'})$ of basepoints. We define
\[
\CFK_n(L) := \overline{\CFK}_n(L')[1 - 1/n];
\]
note that a downward shift by $\frac{1}{2}$ in $\gr_T$ as in \cite[Section 2.2]{dowlin2018family} produces a downward shift by $1 - \frac{1}{n}$ in $\frac{\gr_n}{n} = -\gr_M + 2(1-1/n)\gr_T$.
\end{definition}

We write $\partial_n$ for the differential on either variant of $\CFK_n$; it satisfies $\partial_n^2 = 0$. When $n=1$, the complex $\overline{\CFK}_1(L)$ computes $\widehat{\mathit{HF}}(S^3)$ (see Section~\ref{sec:MasterComplex}), so its homology is $\Q$ in $\gr_M = 0$ (and thus $\frac{\gr_n}{n} = 0$) and zero in other degrees; see \cite[Lemma 5.2]{dowlin2018family}. It follows that $\HFK_1(L)$ is also $\Q$ in degree $0$ and zero in other degrees.

Since the tensor product (after annihilating the final pair of variables) sets each $V_i$ variable equal to a polynomial in the $U_i$ variables while imposing no further relations on the $U_i$ variables, the complexes $\overline{\CFK}_n(L)$ and $\CFK_n(L)$ are free over $\Q[U_1,\ldots,U_{m-1}]$ and $\Q[U_1,\ldots,U_{m'-1}]$ respectively. Their homology groups $\overline{\HFK}_n(L)$ (respectively $\HFK_n(L)$) depend only on $L$ with its distinguished component (respectively, $L$) and are finite-dimensional over $\Q$ as shown in \cite{dowlin2018family}.

\begin{remark}
In \cite{dowlin2018family}, Dowlin uses the notation $\widehat{\HFK}_n(L)$ to refer to what we call $\overline{\HFK}_n(L)$; however, in \cite{dowlin2018spectral}, $\widehat{\HFK}_n$ (at least for $n=2$) is given a different definition which is closer to Definition~\ref{def:HFKHat} for $\widehat{\HFK}(L)$.
\end{remark}

\section{The Euler characteristic of \texorpdfstring{$\HFK_n$}{HFKn}}\label{sec:EulerChar}

Let $L$ be an $\ell$-component link in $S^3$ equipped with a distinguished component; in this section we compute the Euler characteristics of $\overline{\HFK}_n(L)$ and $\HFK_n(L)$. 

For simplicity, assume we are working with a Heegaard diagram $\Hc$ for $L$ that has exactly $2\ell$ basepoints. Let $R = \Q[U_1, ..., U_{\ell-1}]$, a $\frac{1}{n}\Z$-graded ring where $U_i$ has degree $\frac{2}{n}$ as in Section~\ref{sec:HFKnDefs}, and let $K$ be the $\frac{1}{n}\Z$-graded Koszul complex
\[
K = \bigotimes_{i = 1}^{\ell-1} \left(R\left[1 - \frac{2}{n} \right] \xrightarrow{U_i} R \right)
\]
where the tensor products are over $R$.

\begin{lemma}\label{lem:NtoHat}
Let $n \geq 2$. There exists a spectral sequence with each page graded by $\frac{1}{n}\Z$, with differentials of degree $+1$ such that each page is the $\frac{1}{n}\Z$-graded homology of the previous page, and with $E_1$ page $\overline{\HFK}_n(L) \otimes_R K$ and $E_{\infty}$ page $\widehat{\HFK}(L)$ as $\frac{1}{n}\Z$-graded vector spaces.
\end{lemma}

\begin{proof}
The complex $\overline{\CFK}_n(L) \otimes_R K$ can be viewed as a cube of dimension $\ell - 1$ in which each vertex is a copy of $\overline{\CFK}_n(L)$. We equip $\overline{\CFK}_n(L) \otimes_R K$ with a filtration such that every oriented edge of this cube increases the filtration level by 1. Then the differential $d$ on $\overline{\CFK}_n(L) \otimes_R K$ can be decomposed as $d = d_0 + d_1$, where $d_0$ is the differential on each copy of $\overline{\CFK}_n(L)$ and $d_1$ comes from the differential on $K$. 

From this filtration, we get a spectral sequence whose $E_1$ page is $(\overline{\HFK}_n(L) \otimes_R K, (d_1)_*)$. The spectral sequence converges because there are only finitely many nontrivial filtration levels, and the $E_{\infty}$ page is $\widehat{\HFK}(L)$; indeed, the total complex $\overline{\CFK}_n(L) \otimes_R K$ has a contractible subcomplex such that the quotient by this subcomplex is $\widehat{\CFK}(L)$.
\end{proof}

Since $\overline{\HFK}_n(L)$ is finitely generated over $\Q$, the same holds for $\overline{\HFK}_n(L) \otimes_R K$. Thus, the spectral sequence of Lemma~\ref{lem:NtoHat} gives an equality between the $\frac{\gr_n}{n}$-graded Euler characteristics of $\overline{\HFK}_n(L) \otimes_R K$ and $\widehat{\HFK}(L)$. The Euler characteristics of $\overline{\HFK}_n(L) \otimes_R K$ and $\overline{\HFK}_n(L)$ are related by
\[
\chi(\overline{\HFK}_n(L)\otimes_R K) = \left(1 - e^{2 \pi i /n}\right)^{\ell - 1} \chi(\overline{\HFK}_n(L)),
\]
so since $n \geq 2$ we get
\[
\chi(\overline{\HFK}_n(L)) = \left(1 - e^{2\pi i /n}\right)^{1-\ell} \chi(\widehat{\HFK}(L)).
\]
Using Proposition~\ref{prop:BigradedHFKEuler}, we can compute the $\frac{\gr_n}{n}$-graded Euler characteristic of $\widehat{\HFK}(L)$ as follows:
\begin{align*}
\chi(\widehat{\HFK}(L)) &= \sum_{\alpha \in \frac{1}{n}\Z} e^{\pi i \alpha} \dim_{\Q} \left( \widehat{\HFK}(L)_{\frac{\gr_n}{n} = \alpha} \right) \\
&= \sum_{I \in \frac{1}{2}\Z, J \in \Z} e^{\pi i (-J + 2(1-1/n)I)} \dim_{\Q} \left( \widehat{\HFK}(L)_{\gr_T = I, \gr_M = J} \right) \\
&= \left( \sum_{I \in \frac{1}{2}\Z, J \in \Z} (-1)^J t^I \dim_{\Q} \left( \widehat{\HFK}(L)_{\gr_T = I, \gr_M = J} \right) \right) \Bigg|_{t^{1/2} = e^{\pi i (1 - 1/n)} = -e^{-\pi i / n}} \\
&= \left( (-1)^{\ell - 1} t^{\frac{\ell - 1}{2}} (1 - t^{-1})^{\ell - 1} \Delta_L(t) \right) \bigg|_{t^{1/2} = -e^{- \pi i / n}} \\
&=  e^{\pi i (1-\ell) / n} \left(1 - e^{2 \pi i / n}\right)^{\ell - 1} \Delta_L(t)|_{t^{1/2} = -e^{-\pi i / n}}.
\end{align*}

\begin{corollary}
For $n \geq 2$, the $\frac{\gr_n}{n}$-graded Euler characteristic of $\overline{\HFK}_n(L)$ is 
\[
e^{\pi i (1 - \ell) / n} \Delta_L(t)|_{t^{1/2} = - e^{- \pi i / n}}.
\]
\end{corollary}

Since $\HFK_n(L)$ is defined as a grading shift of $\overline{\HFK}_n$ of the disjoint union of $L$ with a split unknot, and the Alexander polynomial vanishes on split links, we see that if $n \geq 2$, the $\frac{\gr_n}{n}$-graded Euler characteristic of $\HFK_n(L)$ is zero for all links $L$. When $n = 1$, we have
\[
\chi(\overline{\HFK}_1(L)) = \chi(\HFK_1(L)) = 1
\]
for all links $L$.

\begin{definition}\label{def:GradingShiftedHFKn}
As in Definition~\ref{def:GradingShiftedHFK}, we define grading-shifted variants $\overline{\HFK}'_n(L)$ and $\HFK'_n(L)$ of $\overline{\HFK}_n(L)$ and $\HFK_n(L)$. Starting with bigradings on $\overline{\CFK}_n(L)$ and $\CFK_n(L)$ corresponding to the bigradings on $\overline{\CFK}'(L)$ and $\CFK'(L)$, the differentials $\partial_n$ have degree $+1$ with respect to
\[
\frac{\gr_n}{n} := \gr_M - 2\left(1 - \frac{1}{n}\right) \gr_T
\]
(note that since we still want $+1$ differentials on $\frac{1}{n}\Z$-graded complexes, this is the negative of the earlier definition of $\frac{\gr_n}{n}$ in terms of $\gr_T$ and $\gr_M$). We define $\overline{\CFK}'_n(L)$ to be $\overline{\CFK}'(L)$ with grading given by $\frac{\gr_n}{n}$ and differential given by $\partial_n$; we define $\CFK'_n(L)$ similarly. 
\end{definition}

\begin{remark}
Starting from $\overline{\HFK}(L)$, we negated both gradings and shifted the Alexander grading upward by $\frac{\ell - 1}{2}$ to get $\overline{\HFK}'(L)$, then applied the collapse $\gr_M - 2(1-1/n)\gr_T$ to get the grading on $\overline{\HFK}'_n(L)$. Equivalently, we could first apply the collapse $-\gr_M + 2(1-1/n)\gr_T$ on $\overline{\HFK}(L)$ to get the grading on $\overline{\HFK}_n(L)$, then shift this $\frac{1}{n}\Z$ grading upward by $(1 - \ell)(1 - 1/n)$. In other words, $\overline{\HFK}'_n(L)$ is $\overline{\HFK}_n(L)$ with its $\frac{1}{n}\Z$-grading shifted upward by $(1 - \ell)(1 - 1/n)$; similarly, $\HFK'_n(L)$ is $\HFK_n(L)$ with its $\frac{1}{n}\Z$-grading shifted upward by $(1 - \ell)(1 - 1/n)$.
\end{remark}

\begin{corollary}
For $n \geq 2$, the $\frac{\gr_n}{n}$-graded Euler characteristic of $\overline{\HFK}'_n(L)$ is 
\[
\Delta_L(t)|_{t^{1/2} = -e^{\pi i / n}},
\]
and the $\frac{\gr_n}{n}$-graded Euler characteristic of $\HFK'_n(L)$ is zero.
\end{corollary}

\begin{proof}
For the reduced case, we have
\[
e^{\pi i (1-\ell)(1-1/n)} e^{\pi i (1 - \ell) / n} \Delta_L(t)|_{t^{1/2} = -e^{-\pi i / n}} = (-1)^{1 - \ell} \Delta_L(t)|_{t^{1/2} = -e^{-\pi i / n}} = \Delta_L(t)|_{t^{1/2} = -e^{\pi i / n}}
\]
(using that $\Delta_L(t^{-1}) = (-1)^{\ell - 1}\Delta_L(t)$); for the unreduced case, we have $e^{\pi i (1-\ell)(1-1/n)} \cdot 0 = 0$.
\end{proof}

The proof of \cite[Lemma 2.23]{dowlin2018family} gives us $\frac{1}{n}\Z$-graded spectral sequences from $\overline{\HFK}'(L)$ to $\overline{\HFK}'_n(L)$ and from $\HFK'(L)$ to $\HFK'_n(L)$, where $\frac{\gr_n}{n}$ on $\overline{\HFK}'(L)$ and $\HFK'(L)$ is defined to be $\gr_M - 2(1-1/n)\gr_T$. 

When $n = 1$, the shifted homology groups $\overline{\HFK}'_1(L)$ and $\HFK'_1(L)$ agree with $\overline{\HFK}_1(L)$ and $\HFK_1(L)$ respectively, so their Euler characteristics are both $1$.

\begin{remark}
The arguments in this section can be made simpler in the case of knots ($\ell = 1$), where by \cite[Lemma 2.20]{dowlin2018family}, $\overline{\HFK}_n(L)$ is isomorphic to $\frac{\gr_n}{n}$-graded $\overline{\HFK}(L) = \widehat{\HFK}(L)$. In particular, Lemma~\ref{lem:NtoHat} is unnecessary in this case.
\end{remark}

\section{Euler characteristics and spectral sequences}\label{sec:SpectralSequences}

Dowlin \cite[Conjecture 1.6]{dowlin2018family} conjectures the existence of spectral sequences from $\overline{H}_n(L)$ to $\overline{\HFK}_n(L)$ and from $H_n(L)$ to $\HFK_n(L)$. These sequences are conjectured to respect the $\frac{1}{n}\Z$-gradings, where the $\frac{1}{n}\Z$-grading $\frac{\gr_n}{n}$ on reduced and unreduced $\mathfrak{sl(n)}$ homology is defined by
\[
\frac{\gr_n}{n} = \frac{1}{n}\gr_{Q,n} + \gr_H.
\]
Dowlin works with $n$ times this grading, which we would write as $\gr_{Q,n} + n \gr_H$.\footnote{In Dowlin's notation this grading is called $\mathbf{gr}_n + \frac{n}{2} \mathbf{gr}_v$ (where $\mathbf{gr}_n$ corresponds to our $\gr_{Q,n}$). As specified in \cite[Section 4.2]{dowlin2018family}, the grading $\mathbf{gr}_v$ here is $k = 2 \gr_v$ in the notation of \cite{rasmussen2006some}, where $\gr_v$ corresponds to our $\gr_H$. This accounts for the factor of $\frac{1}{2}$ in Dowlin's formula.} We state the following version of Dowlin's conjectures in terms of the grading-shifted theories $\overline{\HFK}'_n$ and $\HFK'_n$.

\begin{conjecture}\label{conj:DowlinPrime}
Let $L$ be a link in $S^3$. There exist spectral sequences with each page graded by $\frac{1}{n}\Z$, with differentials of degree $+1$ such that each page is the $\frac{1}{n}\Z$-graded homology of the previous page, and with
\begin{itemize}
\item $E_2$ page $\overline{H}_n(L)$ and $E_{\infty}$ page $\overline{\HFK}'_n(L)$;
\item $E_2$ page $H_n(L)$ and $E_{\infty}$ page $\HFK'_n(L)$
\end{itemize}
as $\frac{1}{n}\Z$-graded vector spaces.
\end{conjecture}

These spectral sequences would give equalities of $\frac{1}{n}\Z$-graded Euler characteristics
\[
\chi(\overline{H}_n(L)) = \chi(\overline{\HFK}'_n(L)), \quad \chi(H_n(L)) = \chi(\HFK'_n(L)).
\]
We have
\begin{align*}
\chi(\overline{H}_n(L)) &= \sum_{\alpha \in \frac{1}{n}\Z} e^{\pi i \alpha} \dim_{\Q} \left( \overline{H}_n(L)_{\frac{1}{n}\gr_{Q,n} + \gr_H = \alpha} \right) \\
&= \sum_{I, J \in \Z} e^{\pi i (I/n + J)} \dim_{\Q} \left( \overline{H}_n(L)_{\gr_{Q,n} = I, \gr_H = J} \right) \\
&= \left( \sum_{I, J \in \Z} (-1)^J q^I \dim_{\Q} \left( \overline{H}_n(L)_{\gr_{Q,n} = I, \gr_H = J} \right) \right) \Bigg|_{q = e^{\pi i / n}} \\
&= \overline{P}_{n,L}(e^{\pi i / n});
\end{align*}
similarly, $\chi(H_n(L)) = P_{n,L}(e^{\pi i / n})$, which is $0$ for $n \geq 2$ and $1$ for $n = 1$. Thus, for $n \geq 2$ these conjectured spectral sequences can be viewed as categorifications of the equalities
\[
\overline{P}_{n,L}(e^{\pi i / n}) = \Delta_L(t)|_{t^{1/2} = -e^{\pi i /n}}, \quad P_{n,L}(e^{\pi i / n}) = 0.
\]
For $n = 1$, the equalities are $\overline{P}_{1,L}(-1) = 1$ and $P_{1,L}(-1) = 1$ (note that $\overline{P}_{1,L}(q) = P_{1,L}(q) = 1$ in general).

\begin{remark}
Let $n \geq 2$ for simplicity. The spectral sequences of Conjecture~\ref{conj:DowlinPrime}, together with the ones from Sections \ref{sec:RasmussenSeq} and \ref{sec:HOMFLYtoHFK} and the spectral sequences from $\HFK$ to $\HFK_n$, can be organized as shown in Figure~\ref{fig:SpectralSeqs}, following \cite[Figure 1]{dowlin2018family}. The arrows in this figure represent spectral sequences (solid for known, dotted for conjectural); we augment Dowlin's figure by labeled the arrows with their decategorified content.\footnote{In fact, ``decategorified content'' is not immediately defined for the arrows on the right edge of the square; the grading-collapsed versions of $\overline{\HFK}(n)$ and $\HFK(n)$ are not finite-dimensional over $\Q$ so their Euler characteristics as $\frac{1}{n}\Z$-graded complexes are not defined. It is reasonable to guess that if a natural definition of these Euler characteristics did exist, it would give the answer in the figure.}

It is interesting to look at the square formed by the reduced theories; traveling along the left edge and then the bottom edge amounts to starting with $\overline{P}_L(a,q)$, evaluating at $a = q^n$, and then evaluating the result at $q = e^{\pi i / n}$ to get $\overline{P}_L(-1, e^{\pi i / n})$. On the other hand, traveling along the top edge and then the right edge amounts to starting with $\overline{P}_L(a,q)$, evaluating at $a=-1$ and $q = -t^{1/2}$, and then evaluating the result at $t^{1/2} = -e^{\pi i / n}$ to get $\overline{P}_L(-1,e^{\pi i / n})$. This compatibility at the level of Euler characteristics could be a sign of a more elaborate compatibility relationship between the conjectured spectral sequences at the categorified level.
\end{remark}

\begin{figure}
\includegraphics[scale=0.7]{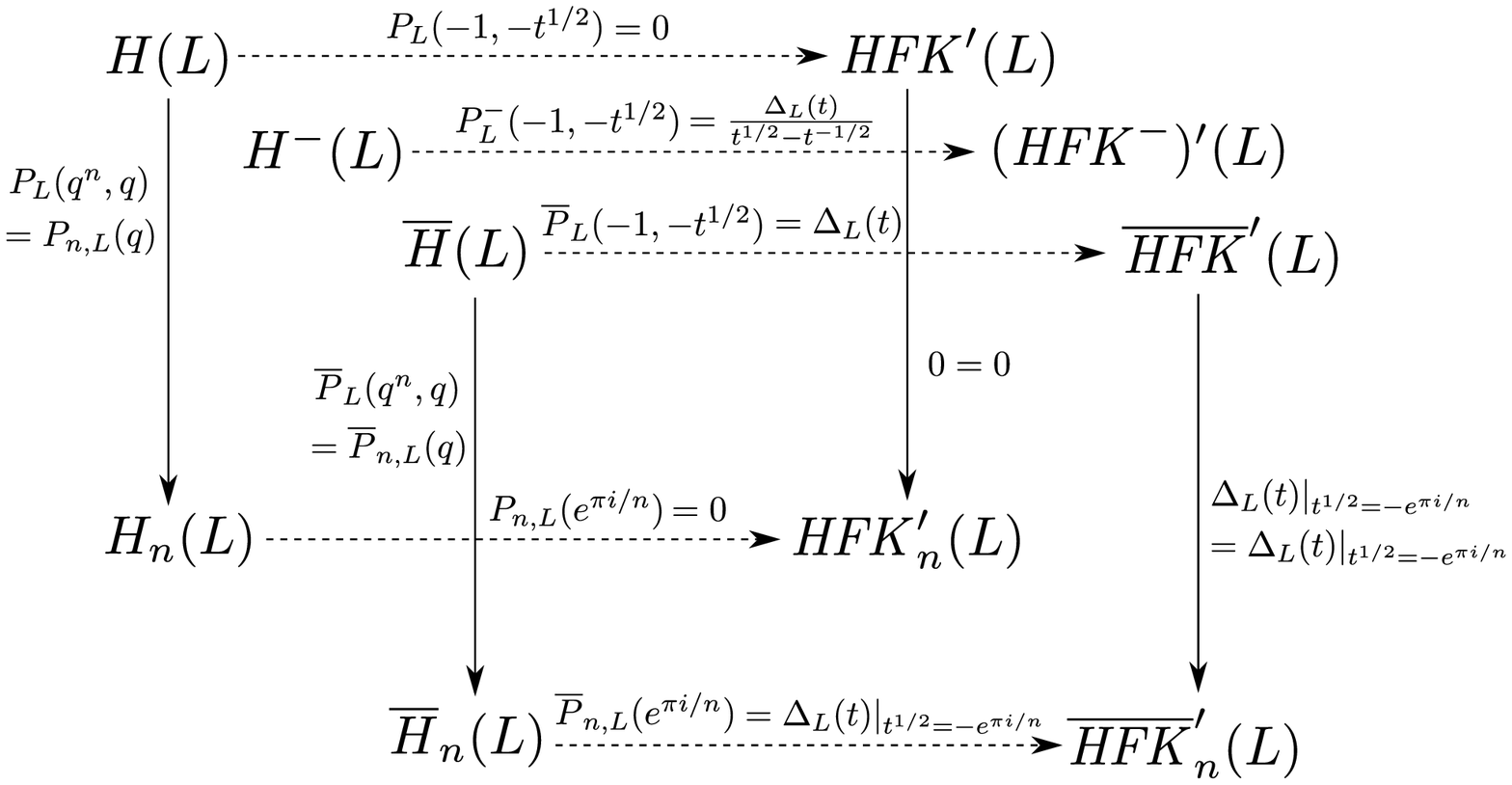}
\caption{Spectral sequences from \cite[Figure 1]{dowlin2018family}, with arrows labeled by decategorified content.}
\label{fig:SpectralSeqs}
\end{figure}

\bibliographystyle{alpha}
\bibliography{HOMFLYcite}

\end{document}